\documentclass[11pt,hidelinks]{article}

\usepackage{orcidlink}

\title{High-dimensional long-range statistical mechanical models \\
have random walk correlation functions}

\author{
Yucheng Liu\,\orcidlink{0000-0002-1917-8330}\thanks{Beijing International Center for Mathematical Research,
	Peking University,
	Beijing, China 100871 and Department of Mathematics,
	University of British Columbia,
	Vancouver, BC, Canada V6T 1Z2.
	\href{mailto:yliu135@pku.edu.cn}{yliu135@pku.edu.cn}.
	}
}
\date{\vspace{-5ex}} 

\usepackage{amsmath, amssymb, amscd, amsthm, amsfonts}
\usepackage{graphicx} 
\usepackage{hyperref} 

\usepackage{booktabs}
\usepackage{xcolor}
\usepackage{ dsfont, bm}
\usepackage[title]{appendix}
\usepackage{comment}
\usepackage{enumerate}
\usepackage{enumitem}
\usepackage{cite}

\usepackage[textwidth=480pt,textheight=650pt,centering]{geometry} 

\usepackage{tikz}
\tikzset{every picture/.style={line width=0.75pt}}

\theoremstyle{plain}
\newtheorem{theorem}{Theorem}[section]
\newtheorem{lemma}[theorem]{Lemma}

\newtheorem{proposition}[theorem]{Proposition}
\newtheorem{corollary}[theorem]{Corollary}

\newtheorem{assumption}[theorem]{Assumption}

\theoremstyle{remark}
\newtheorem{remark}[theorem]{Remark}

\numberwithin{equation}{section}

\newcommand{\ie}{i.e.}

\newcommand{\eps}{\varepsilon}

\newcommand{\Z}{\mathbb{Z}}

\newcommand{\R}{\mathbb{R}}
\newcommand{\C}{\mathbb{C}}

\renewcommand{\P}{\mathbb{P}}
\newcommand{\T}{\mathbb{T}}

\newcommand{\inv}{^{-1}}

\newcommand{\half}{\frac{1}{2}}

\newcommand{\1}{\mathds{1}}

\newcommand{\nl}{\nonumber \\}

\providecommand{\abs}[1]{\lvert#1\rvert}

\providecommand{\biggabs}[1]{\bigg\lvert#1\bigg\rvert}

\providecommand{\norm}[1]{\lVert#1\rVert}

\newcommand{\chip}{{\chi(p)}}

\newcommand{\mup}{{\mu_p}}
\newcommand{\mupc}{{\mu\crit}}

\newcommand{\supSAW}{^{\mathrm{(SAW})}}
\newcommand{\supperc}{^{(\mathrm{perc})}}
\newcommand{\supIsing}{^{(\mathrm{Ising})}}

\newcommand{\veee}[1]{|\!|\!|#1|\!|\!|}

\providecommand{\nnnorm}[1]{\veee {#1}}

\newcommand{\const}{\mathrm{const}}

\newcommand{\crit}{_{p_c}}

\newcommand{\Kub}{K_D}
\newcommand{\Klb}{\Kub\inv}
\newcommand{\tildeD}{P}

\newcommand{\Rd}{{\mathbb R^d}}
\newcommand{\Zd}{\mathbb Z^d}
\newcommand{\Td}{{\mathbb T^d}}

\newcommand{\dc}{d_c}

\begin{document}
\maketitle

\begin{abstract}
We consider long-range percolation, Ising model, and self-avoiding walk on $\mathbb{Z}^d$, with couplings decaying like $|x|^{-(d+\alpha)}$ where $0 < \alpha \le 2$, above the upper critical dimensions.
In the spread-out setting where the lace expansion applies, we show that the two-point function for each of these models exactly coincides with a random walk two-point function, up to a constant prefactor. 
Using this, for $0<\alpha < 2$, we prove upper and lower bounds of the form $|x|^{-(d-\alpha)} \min\{ 1, (p_c - p)^{-2} |x|^{-2\alpha} \}$ for the two-point function near the critical point $p_c$. 
For $\alpha=2$, we obtain a similar upper bound with logarithmic corrections.
We also give a simple proof of the convergence of the lace expansion, assuming diagrammatic estimates. 
\end{abstract}




\section{Introduction and results}
\subsection{Introduction}

Statistical mechanical models on $\Zd$ exhibit mean-field behaviours above an upper critical dimension.
Models with short-range (finite-range or exponentially decaying) couplings have been studied extensively in the literature, 
and the upper critical dimension is $\dc=6$ for percolation
and is $\dc = 4$ for the Ising model and the self-avoiding walk (SAW).
Two main tools to prove mean-field behaviours for percolation and the self-avoiding walk are the lace expansion \cite{BS85,HS90a,Slad06, HHS03, Hara08} and the (reversed) Simon--Lieb inequality \cite{DP25a,DP25b}.
The Ising model offers a wider range of tools, and mean-field results in dimensions $d > \dc$ have been obtained using the lace expansion \cite{Saka07},
correlation inequalities \cite{Froh82},
and methods based on the random current representation \cite{DP25-Ising,Aize82,AF86}; the last is also used to prove the triviality of the Ising model at $d = \dc = 4$ \cite{AD21}.

Long-range models have couplings that decay polynomially like $\abs x^{-(d+\alpha)}$ with $\alpha > 0$. These models have upper critical dimension
\begin{equation} \label{eq:dc}
\dc = \dc(\alpha) = \begin{cases}
2 \times \min\{ \alpha,  2\}	&(\text{Ising, SAW}) \\
3 \times \min\{ \alpha,  2\}	&(\text{percolation}) ,
\end{cases}
\end{equation}
and they can exhibit mean-field behaviours in low dimensions if $\alpha > 0$ is sufficiently small. 
Both the lace expansion \cite{HHS08,CS15,CS19}
and the random current method for the Ising model \cite{Pani24} have been extended to this setting. 
For percolation, a different non-perturbative approach to mean-field behaviours has been developed recently in 
\cite{Hutc25, Hutc26I}. 
(See also \cite{Hutc26II, Hutc26III} for behaviours at or below the upper critical dimension.)

Common in the analyses based on the lace expansion or the Simon--Lieb inequality is the idea that an interacting model in dimensions $d > \dc$ 
should be close to a random walk (a free model).
Most previous works compare the model to the simple random walk on the underlying graph;
\cite{BS85,HS90a,Slad06, HHS08} used such comparisons in the Fourier space, 
and \cite{HHS03, Saka07, CS15, DP25a, DP25b} used such comparisons in the physical space.
Despite being a simple idea, the freedom to \emph{choose} a random walk to compare to has not been explored much. 
Two exceptions are \cite{CS19, Liu25Rd}, where suitable random walks are chosen based on the interacting model to ease the analyses.
In this paper, we push the comparison idea to the extreme, by showing that the two-point function $G_p$ for spread-out long-range models exactly coincides with the two-point function $S_\mup$ of some random walk, with a perturbed one-step distribution, when $\alpha \le 2$:
\begin{equation}
G_p(x)=A_p S_\mup(x)
	\qquad(x\in \Zd)
\end{equation}
for some constant $A_p$.
The proof uses the lace expansion and is inspired by (but different from) the random walk used in \cite{CS19}. 
The common idea is to include the effects of interaction (the lace function), which decays faster than $\abs x^{-(d+\alpha)}$ when $\alpha$ is small, into the one-step distribution, 
but our treatment makes the analysis of subcritical two-point functions very straightforward.

Using the exact random walk representation, the proof of convergence of the lace expansion avoids deconvolution theorems as in \cite{HHS03,LS24b,CS15} and is arguably one of the simplest. 
Interestingly, when $\alpha$ is larger so that it is no longer possible to include all effects of interaction into the one-step distribution, the simple deconvolution method of \cite{LS24a, LS24b} can be applied.
The new method we develop in this paper therefore complements \cite{LS24b} and covers all remaining values of $\alpha > 0$. 

We also use the random walk representation to prove a uniform estimate for the subcritical two-point function $G_p$ near the critical point $p_c$, of the form
\begin{equation} \label{eq:near_critical}
G_p(x) \asymp
\frac 1 {\abs x^{d-\alpha}} \min\biggl\{ 1 , \frac 1 { (p_c - p)^2 \abs x^{2\alpha}  } \biggr\} 
\end{equation}
when $0 < \alpha < 2$.
We also prove a similar upper bound with logarithmic corrections when $\alpha =2$.
Equation \eqref{eq:near_critical} interpolates between
the $\abs x^{-(d-\alpha)}$ behaviour at criticality \cite{CS15,CS19}
and the $\chip^2 \abs x^{-(d+\alpha)}$ behaviour below criticality \cite{Aoun21}, where $\chip = \sum_{x\in \Zd} G_p(x) \asymp (p_c - p)\inv$ is the \emph{susceptibility}.
This is in contrast with short-range models, for which subcritical two-point functions decay exponentially and \eqref{eq:near_critical} should take the form 
\begin{equation} \label{eq:SR_near_critical}
G_p(x) \asymp
\frac 1 {|x|^{d-2}} \exp\Bigl\{ -\const (p_c-p)^{1/2} |x|  \Bigr\} ,
\end{equation}
with possibly different constants in the upper and lower bounds.
As discussed in \cite{LPS25-universal}, the upper bound of \eqref{eq:near_critical} near criticality can be used to study long-range models on the $d$-dimensional discrete torus
(a box with periodic boundary conditions)
when $d > \dc$, and it is used there to prove that the two-point function 
for the long-range self-avoiding walk on the torus
levels off to a nonzero constant called the ``torus plateau'' at criticality.
For short-range models,
both the upper bound of \eqref{eq:SR_near_critical} and the torus plateau have been proved for percolation \cite{HMS23}, the Ising model \cite{DP25-Ising, LPS25-Ising}, and the self-avoiding walk \cite{Liu25EJP, Slad23_wsaw},
while the lower bound of \eqref{eq:SR_near_critical} remains open for large $x$.
For long-range models, \eqref{eq:near_critical} is only known for percolation when $\alpha < 1$ and $d \ge 1$ \cite{Hutc25};
the lower bound of \eqref{eq:near_critical} is known for the Ising model under the same conditions \cite[Proposition~3.25]{Pani24}.
Our method gives \eqref{eq:near_critical} for all $\alpha < 2$ in dimensions $d > \dc$.

\smallskip \noindent \textbf{Notation.}
We write $f \lesssim g$ to mean that there is a $C>0$ such that $f \le Cg$, and we write $f \lesssim_L g$ if the constant depends on a parameter $L$.
We write $f \asymp_{(L)} g$ to mean that $f \lesssim_{(L)} g$ and $g \lesssim_{(L)} f$, where $L$ may or may not be present.
We write $f\sim g$ to mean that $\lim f/g=1$.

We write $a \vee b = \max \{ a , b \}$ and $a \wedge b = \min \{ a , b \}$.
With $L\ge 1$ and $|x|$ the Euclidean norm of $x\in \R^d$, 
we define
\begin{equation} \label{eq:nnnorm}
\nnnorm x_L = \frac \pi 2 \max\{ \abs x , L\}.
\end{equation}
Note that \eqref{eq:nnnorm} does not define a norm on $\R^d$
but satisfies $\nnnorm x_L = L \nnnorm{ \frac x L }_1$. 
Also note that $\log \nnnorm{ \frac x L }_1 \ge \log \frac \pi 2 > 0$ for all $x$.

\subsection{The models}

We consider percolation, the Ising model, and the self-avoiding walk with long-range coupling, described by a probability distribution $D$ on $\Zd$ that satisfies the following assumption.
The spread-out parameter $L \ge 1$ will be taken large.

\begin{assumption} \label{ass:D}
Let $L\ge 1$ and $\alpha > 0$.
We assume $D: \Zd \to [0,\infty)$ is a probability distribution that is $\Zd$-symmetric (invariant under reflection in coordinate hyperplanes and rotation by $\pi/2$)
and satisfies $D(0) = 0$ and
\begin{equation} \label{eq:D_bounds}
\frac {\Klb L^\alpha} { \nnnorm { x }_L ^{d+\alpha }  }
\le D(x)
\le  \frac {\Kub L^\alpha} { \nnnorm { x }_L ^{d+\alpha }  }
	\qquad (x\ne 0) 
\end{equation}
for some constants $\Kub \ge 1$.
We also assume there is a constant $\Delta \in (0,1)$ for which
the Fourier transform $\hat D(k) = \sum_{x\in\Z^d}D(x) e^{ik\cdot x}$ obeys
\begin{equation} \label{eq:Dhat}
1 - \hat D(k) \begin{cases}
< 2 - \Delta	& (k \in (-\pi, \pi]^d) \\
> \Delta		& ( \norm k_\infty > L\inv ) .
\end{cases}
\end{equation}
\end{assumption}

The assumption $D(0) = 0$ can be relaxed but is convenient. 
A class of distributions that obey Assumption~\ref{ass:D} can be defined as follows. 
Let $v : \R^d \to [0, \infty)$ be a $\Zd$-symmetric, piecewise-continuous function that satisfies $\int_\Rd v(y)dy = 1$ and $v(y) \asymp \nnnorm y_1^{-(d+\alpha)}$ for all $y \in \Rd$. 
We define
\begin{equation}
D(x) = \frac{ v(x/L) }{\sum_{x \ne 0} v(x/L)}  \1_{x\ne 0} ,
\end{equation}
then the denominator is asymptotic to $L^d$ by a Riemann-sum approximation, and \eqref{eq:D_bounds} holds. 
Property \eqref{eq:Dhat} is verified in \cite[Proposition~1.1]{CS08} for $L$ sufficiently large.

We now define the models.
Let $D$ obey Assumption~\ref{ass:D}.
We write $\omega = ( \omega_0, \dots, \omega_{\abs \omega} ) \in \bigcup_{n=1}^\infty \mathcal (\Zd)^n$ for a \emph{path} in $\Zd$, where $\abs \omega$ denotes the length of the path. 
For $x\in \Zd$, we write $\mathcal W(x)$ for the set of all paths from $0$ to $x$, and we write $\mathcal S(x)$ for the set of all self-avoiding paths from $0$ to $x$, \ie, the set of paths $\omega$ such that $\omega_i \ne \omega_j$ for all $0 \le i < j \le \abs \omega$.
Given $p \ge 0$ and $x\in \Zd$, 
the \emph{random walk two-point function} and the \emph{self-avoiding walk two-point function} are defined respectively by
\begin{equation} \label{def:C}
C_p(x) = \sum_{\omega \in \mathcal W(x)}
	\prod_{j=1}^{\abs \omega} pD(\omega_j - \omega_{j-1}) , \qquad
G_p\supSAW(x) = \sum_{\omega \in \mathcal S(x)}
	\prod_{j=1}^{\abs \omega} pD(\omega_j - \omega_{j-1}) .
\end{equation}
The sum defining $C_p(x)$ converges for all $p \in [0, 1)$, 
and it converges also at the critical point $p=1$ in dimensions $d > \alpha\wedge 2$.
We will also consider a random walk with a different one-step distribution $\tildeD(x)$, whose two-point function we denote by $S_p(x)$.

For percolation, let $p \in [0, \norm D_\infty \inv ]$.
To each pair of vertices $\{x, y \}$ in $\Zd$, 
we associate independently a Bernoulli random variables $n_{\{x, y \}}$ with
\begin{equation}
\P_p( n_{\{x, y \}} = 1 ) = pD(x-y) .
\end{equation}
A configuration is a realisation of these random variables.
We say the bond $\{ x, y \}$ is \emph{occupied} if $n_{\{x, y \}} = 1$ and is \emph{vacant} otherwise. 
The set of all vertices $x$ that are connected to $0$ by some path of occupied bonds is denoted $\mathcal C(0)$. 
The \emph{percolation two-point function} is defined as
\begin{equation}
G_p\supperc(x) = \P_p( x \in \mathcal C(0)) .
\end{equation}

For the Ising model, 
we assume a ferromagnetic coupling $J: \Zd \to [0,\infty)$ that obeys Assumption~\ref{ass:D} (replace $D$ by $J$), and we first consider a finite subset $\Lambda \subset \Zd$ containing $0$. 
The Ising model on $\Lambda$ is a family of probability measures on spin configurations $\sigma \in  \{ \pm 1\}^\Lambda$,
defined using the Hamiltonian
\begin{equation}
H^\Lambda (\sigma) = - \half \sum_{ x,y\in \Lambda } J(y-x) \sigma_x \sigma_y .
\end{equation}
At inverse temperature $\beta \ge 0$,
the \emph{Ising two-point function} is defined by the infinite-volume limit
\begin{equation}
\tilde G_\beta\supIsing (x)
= \lim_{\Lambda \uparrow \Zd } 
	\langle \sigma_0 \sigma_x  \rangle_\beta^\Lambda
= \lim_{\Lambda \uparrow \Zd } 
	\frac{ \sum_{\sigma \in  \{ \pm 1\}^\Lambda} \sigma_0 \sigma_x e^{-\beta H^\Lambda(\sigma)}  } { \sum_{\sigma \in  \{ \pm 1\}^\Lambda} e^{-\beta H^\Lambda(\sigma)} } ,
\end{equation}
which exists because $\langle \sigma_0 \sigma_x  \rangle_\beta^\Lambda$ is increasing in $\Lambda$ by Griffiths inequalities.
For $\beta > 0$, it is convenient to define a probability distribution $D$ by
\begin{equation} \label{eq:D_Ising}
D(x) = p \inv \tanh( \beta J(x) ) ,
\qquad p = \sum_{y\in \Zd} \tanh( \beta J(y) ) .
\end{equation}
Note that $p$ is strictly increasing and continuous in $\beta$,
so we can write $G_p\supIsing(x) = \tilde G_\beta\supIsing (x)$. 
It follows from $\norm J_\infty \lesssim L^{-d}$ and from the Taylor expansion $\tanh(z) = z + O(z^3)$ that 
\begin{equation} \label{eq:Ising_beta}
\frac p \beta ,\, 
\frac{ D(x) } { J(x) }
= 1 + O(L^{-2d})
\end{equation}
uniformly in $x\ne 0$ and in $0 < \beta \le 4$,
when $L$ is large.
Thus, $D$ satisfies Assumption~\ref{ass:D} with $\beta$-independent constants that are slightly perturbed from those of $J$.

We omit the model labels on $G_p$ from now on. 
For each of the models, there is a critical point $p_c \ge 1$ such that $\chi(p) = \sum_{x\in\Zd} G_p(x)$ is finite if and only if $p \in [0,p_c)$.
Under a slightly stronger assumption on $D$, 
Chen and Sakai \cite{CS15,CS19} proved that, 
in dimensions $d > \dc$ (for $\alpha \ne 2$) or $d \ge \dc$ (for $\alpha = 2$), the critical two-point function $G\crit(x)$ is finite and behaves as $\abs x\to \infty$ as
\begin{equation}
G\crit(x) \sim \frac{ \const_L }{ \abs x^{d - \alpha \wedge 2} }
\times \begin{cases}
1 				& (\alpha \ne 2) \\
( \log \abs x )\inv	& ( \alpha = 2 ) .
\end{cases} 
\end{equation}
In particular, this implies finiteness of the bubble diagram $B(p_c) = G\crit * G\crit(0)$ for Ising and SAW and finiteness of the triangle diagram $T(p_c) = G\crit^{*3}(0)$ for percolation, so that the critical exponents $\gamma, \beta, \delta$ take their mean-field values \cite{BFF84, AN84, BA91, Aize82, AF86} (see also \cite[Appendix~A]{HHS08}). 
Note the logarithmic correction for $\alpha = 2$ makes the 
bubble/triangle diagram finite
also at the critical dimension $\dc$.

\subsection{Main results}

Our main result is the following theorem, which represents the two-point function for percolation, the Ising model, or the self-avoiding walk as a random walk two-point function. 
Recall the $\alpha$-dedpendent upper critical dimension $\dc$ in \eqref{eq:dc}.
Let
\begin{equation} \label{eq:ell}
\ell = \begin{cases}
3	&(\text{Ising, SAW}) \\
2	&(\text{percolation}).
\end{cases}
\end{equation}

\begin{theorem} \label{thm:G=S}
Let $\alpha, d > 0$ be such that $d > \dc$
and $\alpha \le 2 + (\ell - 1)(d - \dc)$.
If $\alpha = 2$, we furthur allow $d = \dc$.
Let $D$ (for percolation or SAW) or $J$ (for Ising) obey Assumption~\ref{ass:D},
with $L \ge L_0$ where $L_0$ is sufficiently large.
Then
for all $p\in [\half, p_c]$, 
there are 
 $A_p > 0$, $\mu_p \in [0,1]$,
and a probability distribution $\tildeD$ depending on $p$
such that
\begin{equation} \label{eq:G=S}
G_p(x) = A_p S_\mup (x)
	\qquad (x\in \Zd) ,
\end{equation}
where $S_\mu$ denotes the random walk two-point function generated by $\tildeD$.
Moreover, we have
\begin{equation} \label{eq:perturb}
A_p, \,    \frac \mup p = 1 + O(L^{-d}) ,
\qquad
\frac{ \tildeD(x) }{ D(x) }  = 1+ O(L^{- (\ell-1)d } )
\end{equation}
uniformly in $p \in [\half, p_c]$ and in $x\ne 0$.
We also have
$\mupc = 1$ and 
\begin{equation} \label{eq:near_crit_claim}
1 - \mup \asymp p_c - p .
\end{equation}
\end{theorem}

Theorem~\ref{thm:G=S} gives a way to study $G_p$ by studying a random walk with a slightly perturbed kernel $\tildeD$ that still obeys Assumption~\ref{ass:D}. 
For example, when $0 < \alpha < 2$, from \cite{BL02} we have the heat kernel estimate
\begin{equation} \label{eq:heat_kernel}
\tildeD^{*n}(x) \asymp_L  
\bigg( \frac 1 { n^{d/\alpha} } \wedge \frac n {\abs x^{d+\alpha}} \bigg)
	\qquad (n\ge 2),
\end{equation}
where $\tildeD^{*n}$ denotes the $n$-fold convolution of $\tildeD$ with itself.
(Note \eqref{eq:heat_kernel} also holds for $x\ne 0$ when $n=1$.)
A short computation\footnote
{
The upper bound follows as in \eqref{eq:ub_pf1} and \eqref{eq:ub_pf2} (set $L=1$ there).
The lower bound is trivial for $x=0$. 
For $x\ne 0$, we use \eqref{eq:heat_kernel} 
and $\mu_p^n \ge e\inv$
when $1 \le n \le \abs x^{\alpha} \wedge (-\log\mup)\inv =: N$, 
to get
\begin{equation*}
S_\mup(x) \gtrsim_L  
	\frac 1 {\abs x^{d+\alpha}} \sum_{n=1}^N n \mu_p^n
\ge \frac 1 {\abs x^{d+\alpha}} \frac{ N^2 }{2e}
= \frac 1 {2e \abs x^{d-\alpha} }
	\bigg( 1 \wedge \frac 1 { \abs{ x }^{2\alpha} (-\log \mup)^2 } \bigg) .
\end{equation*}
The result then follows from $-\log \mup \asymp 1-\mup \asymp p_c - p$.
}
then gives
\begin{equation} \label{eq:near_critical_asymp}
G_p(x) = A_p S_\mup(x) 
\asymp_L \frac 1 {\nnnorm x _1^{d-\alpha} } 
	\bigg( 1 \wedge \frac 1 { \nnnorm{ x }_1^{2\alpha}(p_c-p)^2 } \bigg) ,
\end{equation}
as claimed in \eqref{eq:near_critical}.

The proof of Theorem~\ref{thm:G=S} uses a similar upper bound on the heat kernel from \cite{CS15,CS19}, which keeps track of the dependence on $L$ also.
It works for all $\alpha > 0$ and implies the following near-critical upper bound for the random walk two-point function.

\begin{proposition} \label{prop:C}
Let $\alpha > 0$, $d > \alpha \wedge 2$,
and let $D$ obey Assumption~\ref{ass:D}.
If $\alpha \ne 2$,
there is a constant $K_C > 0$ depending only on $\Kub$ and $\Delta$ such that for all $\mu \in [0,1]$ and $x\in \Zd$,
\begin{equation} \label{eq:C_bound}
C_\mu(x) \le \delta_{0,x} +
	\frac{ K_C }{ L^{\alpha\wedge 2} \nnnorm x_L^{d - \alpha \wedge 2 } }
	\bigg( 1 \wedge \frac 1 { \nnnorm{ \frac x L }_1^{2(\alpha\wedge2)} (1-\mu)^2 } \bigg) .
\end{equation}
If $\alpha = 2$, we instead have
\begin{equation} \label{eq:C_bound2}
C_\mu(x) \le \delta_{0,x} +
	\frac{ K_C }{ L^2 \nnnorm x_L^{d - 2 } \log \nnnorm {\frac x L}_1 }
	\bigg( 1 \wedge \frac { ( \log \nnnorm {\frac x L}_1)^2 } { \nnnorm{ \frac x L }_1^4 (1-\mu)^{2} } \bigg) .
\end{equation}
\end{proposition}

In view of \eqref{eq:perturb},
Theorem~\ref{thm:G=S} and Proposition~\ref{prop:C} immediately give the same near-critical upper bound for $G_p$.
When $0<\alpha<2$, the upper bound refines \eqref{eq:near_critical_asymp}.

\begin{corollary} \label{cor:near_crit}
Under the hypotheses of Theorem~\ref{thm:G=S},
if $\alpha \ne 2$
there is a constant $K > 0$ independent of $L$ such that for all $p \in [\half,p_c]$ and $x\in \Zd$,
\begin{equation}
G_p(x) \le \delta_{0,x} +
	\frac{ K }{ L^{\alpha\wedge 2} \nnnorm x_L^{d - \alpha \wedge 2 } }
	\bigg( 1 \wedge \frac 1 { \nnnorm{ \frac x L }_1^{2(\alpha\wedge2)} (p_c - p)^2 } \bigg) .
\end{equation}
If $\alpha = 2$, we instead have
\begin{equation}
G_p(x) \le \delta_{0,x} +
	\frac{ K }{ L^2 \nnnorm x_L^{d - 2 } \log \nnnorm {\frac x L}_1 }
	\bigg( 1 \wedge \frac { ( \log \nnnorm {\frac x L}_1)^2 } { \nnnorm{ \frac x L }_1^4 (p_c - p)^2 } \bigg) .
\end{equation}
\end{corollary}

\begin{proof}
Since $G_p(0) = 1$ for all our models, it suffices to prove the bound for $x\ne 0$. 
We prove for $\alpha \ne 2$; the proof for $\alpha = 2$ differs only in using \eqref{eq:C_bound2} instead of \eqref{eq:C_bound}.
By Theorem~\ref{thm:G=S} and Proposition~\ref{prop:C} 
with $\tildeD$, 
for $x\ne 0$ we have
\begin{equation} \label{eq:near_crit_pf1}
G_p(x) = A_p S_\mup(x)
\le  \frac{ A_p K_S }{ L^{\alpha\wedge 2} \nnnorm x_L^{d - \alpha \wedge 2 } }
	\bigg( 1 \wedge \frac 1 { \nnnorm{ \frac x L }_1^{2(\alpha\wedge2)} (1-\mup)^2 } \bigg) .
\end{equation}
From \eqref{eq:near_crit_claim}, we have $1 - \mup \asymp p_c - p$.
From \eqref{eq:perturb}, we have $A_p = 1 + O(L^{-d})$ uniformly in $p$.
Also, it follows from the uniform estimate on $\tildeD(x) / D(x)$ that $\tildeD$ obeys Assumption~\ref{ass:D} with constants independent of $p$ (they can be chosen to be $1+ O(L^{- (\ell-1)d})$ times the corresponding constants for $D$),
so $K_S$ is also independent of $p$. 
(Alternatively, one can use a slight extension of Proposition~\ref{prop:C} given in Proposition~\ref{prop:S} below.)
Therefore, we can bound $A_p K_S \le K$ and conclude the proof.
\end{proof}

\subsection{Lace expansion}

The proof of Theorem~\ref{thm:G=S} uses the lace expansion for the self-avoiding walk \cite{BS85,Slad06}, percolation \cite{HS90a}, and the Ising model \cite{Saka07}.
The lace expansion is a perturbative method that produces a convolution equation for $G_p$, and its convergence requires a small parameter, which is $L\inv$ in our spread-out setting and is $(d-\dc)\inv$ for nearest-neighbour models. We refer to \cite{Slad06} for an introduction to the lace expansion method. 
For the self-avoiding walk, the lace expansion equation is
\begin{equation} \label{eq:lace}
G_p = \delta + pD * G_p + \Pi_p * G_p
	\qquad(0 \le p \le p_c) ,
\end{equation}
where $\delta(x) = \delta_{0,x}$ is the Kronecker delta
and $\Pi_p : \Zd \to \R$ is an explicit but complicated function encoding the self-avoiding constraint. 
Indeed, if $\Pi_p$ is absent, the equation reduces to the renewal equation
\begin{equation} \label{eq:C}
C_p = \delta + pD * C_p
	\qquad (0 \le p\le 1)
\end{equation}
for the random walk two-point function $C_p$.
For percolation and the Ising model, the lace expansion equation is different from \eqref{eq:lace}, but it is a recent observation of \cite{LS24b} that the equation can be transformed into \eqref{eq:lace} in the spread-out setting (see Section~\ref{sec:lace}).
This transformation is crucial for the exact random walk representation.

Common to previous works \cite{HHS03,CS15,CS19,LS24b} based on the lace expansion, we employ a bootstrap argument in $x$ space to prove that, for $L$ sufficiently large
\begin{equation} \label{eq:Pi_claim}
\abs{ \Pi_p(x) }
\lesssim \bigg( \frac 1 { L^{\alpha\wedge2} \nnnorm x_L^{d-\alpha\wedge 2} } \bigg)^\ell 
= \bigg( \frac 1 { L^d \nnnorm{ \frac x L }_1^{d-\alpha\wedge 2} } \bigg)^\ell 
	\qquad (x\ne 0) ,
\end{equation}
where $\ell \in \{2,3\}$ is defined in \eqref{eq:ell}.
The bound is due to geometric considerations of the interaction. 
For example, for percolation, a main contribution to $\Pi_p(x)$ is from the event that $0$ and $x$ are connected by two disjoint paths (doubly-connected), so the BK inequality implies that it decays at least twice ($\ell=2$) as fast as the two-point function $G_p(x)$. 
We do not have new things to say about how $\Pi_p(x)$ is bounded,
but our bootstrap step is simpler than all previous works, since we do not need to use deconvolution theorems, thanks to the exact random walk representation \eqref{eq:G=S}.

\subsection{Strategy of proof}
\label{sec:strategy}

We now present the strategy of the proof of Theorem~\ref{thm:G=S} for $0 < \alpha \le 2$. 
The extension to all $\alpha \le 2 + (\ell - 1)(d - \dc)$ will be discussed near the end of Section~\ref{sec:strategy}.

The main idea is to define a new distribution $\tildeD$ proportional to $pD + \Pi_p$, then \eqref{eq:lace} becomes like \eqref{eq:C}, and we can conclude the theorem using the uniqueness of solutions to the convolution equation. 
If $D(x)$ decays too fast (e.g., nearest-neighbour), this strategy does not work because $\Pi_p(x)$ can take negative values.
However, if $\alpha$ is small, $D(x)$ has a fat tail and might overtake any negative parts of $\Pi_p(x)$, to make $\tildeD(x)$ positive everywhere.\footnote{
We are inspired by a similar observation made on $(\delta + \Pi_p) * D$ by Chen and Sakai in \cite[Remark~3.7]{CS19}. See also Remark~\ref{rmk:CS19}.
}
This is indeed the case when $\alpha \le 2$.
We combine the decay of $\Pi_p(x)$ in \eqref{eq:Pi_claim} with the assumed lower bound of $D(x)$ in \eqref{eq:D_bounds}, to get
\begin{equation}
(pD + \Pi_p)(x) \gtrsim 
	\frac 1 {L^d}  \bigg(  \frac 1 { \nnnorm{ \frac x L }_1 ^{d+\alpha } } 
	- \frac{ O(L^{-(\ell-1)d } ) }{ \nnnorm{ \frac x L }_1^{\ell ( d- \alpha) }} \bigg) 
	\qquad (x\ne 0)
\end{equation}
uniformly in $p \in [\half, p_c]$.
Since the exponents of $\nnnorm{ \frac x L }_1$ satisfy
\begin{equation} \label{eq:decay_comp}
\ell (d-\alpha) 
= d + \alpha + [ ( \ell - 1 ) d - ( \ell + 1 ) \alpha ]
\ge d + \alpha
\end{equation}
when $d \ge \frac{\ell + 1 }{ \ell - 1} \alpha = \dc $ (recall \eqref{eq:dc} and \eqref{eq:ell}), we get
\begin{equation}
(pD + \Pi_p)(x)
\gtrsim  \frac 1 {L^d  \nnnorm{ \frac x L }_1 ^{d+\alpha } } 
	[ 1 -  O(L^{-(\ell-1)d } )  ] 
\ge 0
\end{equation}
when $L$ is sufficiently large.
This allows us to define a probability distribution $\tildeD$ by $\tildeD(0)=0$ and
\begin{equation} \label{eq:tilde_D}
\tildeD(x)
= \frac{ pD(x) + \Pi_p(x) } { \norm{  (pD + \Pi_p) \1_{x\ne 0} }_1  }
= \frac{ pD(x) + \Pi_p(x) } { p + \sum_{y\ne 0} \Pi_p(y) }
	\qquad (x\ne 0) .
\end{equation}
Then, by decomposing $\Pi_p (x) = \Pi_p(0) \delta_{0,x} + \Pi_p(x) \1_{x\ne 0}$ in equation \eqref{eq:lace}, 
we get
\begin{align} \label{eq:mup}
[1 - \Pi_p(0)] G_p &= \delta + [( pD + \Pi_p ) \1_{x\ne0} ] * G_p  \nl
&= \delta + \underbrace{
	\bigg( \frac{ p + \sum_{y\ne 0} \Pi_p(y) }{ 1 - \Pi_p(0)  } \bigg)
	}_{= : \mup}
	\tildeD * [1 - \Pi_p(0)] G_p .
\end{align}
Comparing to the random walk equation \eqref{eq:C} with kernel $\tildeD$, 
once we prove that $\mup \in [0,1]$, uniqueness then yields
\begin{equation}
[1 - \Pi_p(0)] G_p = S_\mup,
\end{equation}
where $S_\mu$ is the random walk two-point function generated by $\tildeD$. The desired \eqref{eq:G=S} follows, with $A_p = [1 - \Pi_p(0)]\inv$.

\begin{remark} \label{rmk:CS19}
In \cite[(3.38)]{CS19}, Chen and Sakai derived another representation of the two-point function, as
\begin{equation} \label{eq:G-CS19}
G_p(x) = ( ( \delta + \tilde \Pi_p ) * \tilde S_{\tilde \mu_p}) (x) 
	\qquad(x\in \Zd) ,
\end{equation}
where $\tilde \Pi_p, \tilde S_{\tilde \mu_p}$ are similar to our $\Pi_p, S_{\mu_p}$ respectively. 
If one wishes to prove \eqref{eq:near_critical_asymp} or Corollary~\ref{cor:near_crit} from \eqref{eq:G-CS19}, Proposition~\ref{prop:C} can still be applied to $\tilde S_{\tilde \mu_p}$, 
but a more involved estimate on $\tilde \Pi_p$ that keeps track of its dependence on $p$ would be necessary. 
One would also need to establish suitable subcritical convolution estimates like \cite[Lemma~8]{LPS25-universal}.
Our method avoids such technicality and requires only the $p$-independent bound in \eqref{eq:Pi_claim}.
\end{remark}

For $\alpha > 2$, the strategy still works if $D(x)$ decays slower than $\Pi_p(x)$. By \eqref{eq:Pi_claim}, this imposes
\begin{equation}
\ell (d - 2) \ge d+\alpha .
\end{equation}
Using the fact that $\ell (\dc - 2) = \dc + 2$,
this is equivalent to the condition
\begin{equation} \label{eq:alpha_cond}
\alpha  \le \ell(d-2) - d  - \ell (\dc - 2) + (\dc + 2)
= 2 + (\ell - 1)(d - \dc) ,
\end{equation}
which we assume. 

The full proof of Theorem~\ref{thm:G=S} is carried out in Section~\ref{sec:G=S}. As shown above, the key step is to prove \eqref{eq:lace} and \eqref{eq:Pi_claim} for all models. 
In the bootstrap argument there, we use an extended version of Proposition~\ref{prop:C}, which we prove first in Section~\ref{sec:S}.
The proof of $\mupc = 1$ uses a tool developed in the Appendix, which bypasses the left-continuity of $\Pi_p$ at $p_c$.

\section{Random walk two-point function}
\label{sec:S}

In this section, we prove a stronger version of Proposition~\ref{prop:C} which also estimates random walks whose kernel $\tildeD(x)$ is close to $D(x)$. 
With $\mu \in [0,1]$, we write $S_\mu$ for the random walk two-point function generated by $\tildeD$, defined as in \eqref{def:C}.
By partitioning according to the length $\abs \omega$, we have
\begin{equation} \label{eq:S_mu}
S_\mu(x) = \delta_{0,x} + \sum_{n=1}^\infty \mu^n \tildeD^{*n}(x) .
\end{equation}
The following proposition implies Proposition~\ref{prop:C}, by taking $P = D$.

\begin{proposition} \label{prop:S}
Let $\alpha > 0$, $d > \alpha \wedge 2$,
and let $D$ obey Assumption~\ref{ass:D}.
There is a constant $K_C > 0$ depending only on $\Kub$ and $\Delta$ such that: 
Whenever $P: \Zd \to [0, \infty)$ is a $\Zd$-symmetric probability distribution such that $P(0) = 0$,
\begin{equation} \label{eq:P_bounds}
\frac { \half \Klb L^\alpha} { \nnnorm { x }_L ^{d+\alpha }  }
\le P(x)
\le  \frac {2 \Kub L^\alpha} { \nnnorm { x }_L ^{d+\alpha }  }
	\qquad (x\ne 0) ,
\end{equation}
and
\begin{equation} \label{eq:Phat}
1 - \hat P(k) \begin{cases}
< 2 - \half \Delta	& (k \in (-\pi, \pi]^d) \\
> \half \Delta		& ( \norm k_\infty > L\inv ) ,
\end{cases}
\end{equation}
we have
\begin{equation} \label{eq:S_bound}
S_\mu(x) \le \delta_{0,x} +
	\frac{ K_C }{ L^{\alpha\wedge 2} \nnnorm x_L^{d - \alpha \wedge 2 } }
	\bigg( 1 \wedge \frac 1 { \nnnorm{ \frac x L }_1^{2(\alpha\wedge2)} (1-\mu)^2 } \bigg) 
\qquad (\alpha \ne 2)
\end{equation}
or
\begin{equation} \label{eq:S_bound2}
S_\mu(x) \le \delta_{0,x} +
	\frac{ K_C }{ L^2 \nnnorm x_L^{d - 2 } \log \nnnorm {\frac x L}_1 }
	\bigg( 1 \wedge \frac { ( \log \nnnorm {\frac x L}_1)^2 } { \nnnorm{ \frac x L }_1^4 (1-\mu)^{2} } \bigg) 
~\qquad (\alpha = 2)
\end{equation}
for all $\mu \in [0,1]$ and $x\in \Zd$.
\end{proposition}

We prove Proposition~\ref{prop:S} by inserting the following estimates on $\tildeD^{*n}$ into \eqref{eq:S_mu}.
The lemma refines the upper bound of \eqref{eq:heat_kernel}.

\begin{lemma} \label{lem:heat_kernel}
Let $\alpha > 0$, $d > \alpha \wedge 2$,
and let $P$ be a $\Zd$-symmetric probability distribution that obeys \eqref{eq:P_bounds} and \eqref{eq:Phat}.
There are constants depending only on $\Kub$ and $\Delta$ such that for all $n\ge 1$ and all $x\in \Zd$,
\begin{align}
\label{eq:Dnsup}
\norm{\tildeD^{*n}}_\infty &\lesssim \frac{1}{L^d} \times
	\begin{cases}
	n^{ - d  / ( \alpha \wedge 2 ) } 	& (\alpha \ne 2) \\
	( n \log \frac {\pi n}2 )^{-d/2}		& ( \alpha = 2 ) ,
	\end{cases} \\
\label{eq:Dnx}
\tildeD^{*n}(x)  &\lesssim  \frac{ n L^{\alpha \wedge 2 } }{ \nnnorm x_L^{d + \alpha \wedge 2 } } \times
	\begin{cases}
	1 						& (\alpha \ne 2) \\
	\log \nnnorm {\frac x L}_1	& ( \alpha = 2 ) .
	\end{cases} 	
\end{align}
\end{lemma}

\begin{remark}
For $\alpha > 2$ and $d > 2$,
if \eqref{eq:Dnx} is improved to 
$\tildeD^{*n}(x)  \lesssim   n^{\alpha/2} L^{\alpha} / \nnnorm x_L^{d + \alpha } $, the proof of Proposition~\ref{prop:S} would give
\begin{equation}
S_\mu(x) \le \delta_{0,x} +
	\frac{ K_C }{ L^{2} \nnnorm x_L^{d - 2} }
	\bigg( 1 \wedge \frac 1 { \nnnorm{ \frac x L }_1^{2+\alpha} (1-\mu)^{(2+\alpha)/2} } \bigg) ,
\end{equation}
which should be sharp.
\end{remark}

\begin{proof}[Proof of Lemma~\ref{lem:heat_kernel}]
These bounds have been obtained for $\tildeD  = D$ under Assumption~\ref{ass:D} in \cite{CS08, CS15,CS19}. 
The proof uses only the estimates \eqref{eq:D_bounds}--\eqref{eq:Dhat} for $D$, so it applies equally to $\tildeD$ satisfying \eqref{eq:P_bounds}--\eqref{eq:Phat}
and produces constants independent of $\tildeD$. 

For \eqref{eq:Dnsup}, we first use the upper and lower bounds in \eqref{eq:P_bounds} to get
\begin{equation}
1 - \hat \tildeD(k) 
\asymp ( L \abs k)^{\alpha \wedge 2} \times 
	\begin{cases}
	1 						& (\alpha \ne 2) \\
	\log \frac \pi {2L \abs k}		& ( \alpha = 2 ) 
	\end{cases} 
\end{equation}
for $\abs k \le L\inv$, from \cite[(A.6)--(A.12)]{CS08} (take $h(y) = \nnnorm y_1^{-(d+\alpha)}$ and $\ell =1$).
Then, in conjunction with \eqref{eq:Phat} and $\norm \tildeD_\infty \lesssim L^{-d}$ from \eqref{eq:P_bounds}, 
from \cite[(A.1)--(A.4)]{CS08} (for $\alpha \ne 2$) or \cite[Section~2.2]{CS19} (for $\alpha=2$) we obtain \eqref{eq:Dnsup}.
For \eqref{eq:Dnx}, only the upper bound of $\tildeD(x)$ is needed, and the argument is given around \cite[(1.21)]{CS15}.
This concludes the proof.
\end{proof}

\begin{proof}[Proof of Proposition~\ref{prop:S}]
We prove the bound for $S_1$ first. 
Since $S_\mu \le S_1$, this will also give the ``$1$'' part of the minimum in \eqref{eq:S_bound}--\eqref{eq:S_bound2} for $S_\mu$.
We follow \cite[(2.3)]{CS15} (for $\alpha\ne 2$) and \cite[(2.1)]{CS19} (for $\alpha = 2$).
If $\alpha < 2$, we set $N = \nnnorm{ \frac x L }_1^\alpha = (\nnnorm x_L / L)^\alpha$. For $\tildeD^{*n}(x)$ in \eqref{eq:S_mu}, we use \eqref{eq:Dnx} if $n \le N$ and use \eqref{eq:Dnsup} if $n \ge N$. This gives
\begin{equation} \label{eq:ub_pf1}
S_1(x) - \delta_{0,x} 
\lesssim \frac{ L^{\alpha} }{ \nnnorm x_L^{d + \alpha } } \sum_{n=1}^N n
	+ \frac{1}{L^d} \sum_{n=N}^\infty	n^{ - d  / \alpha }
\lesssim \frac{ L^{\alpha} N^2 }{ \nnnorm x_L^{d + \alpha } }
	+ \frac{ N^{1-d/\alpha} }{L^d}
= \frac{ 2 }{ L^\alpha \nnnorm x_L^{d - \alpha } } ,
\end{equation}
as desired. 
The same proof works for $\alpha > 2$ after replacing all $\alpha$ by $2$. 
If $\alpha = 2$, 
we set $N = \nnnorm{ \frac x L }_1^2 / \log \nnnorm{ \frac x L }_1$ instead,
and we have
\begin{align}
S_1(x) - \delta_{0,x} 
&\lesssim \frac{ L^{2} \log \nnnorm{ \frac x L}_1 }{ \nnnorm x_L^{d + 2 }  } \sum_{n=1}^N n
	+ \frac{1}{L^d} \sum_{n=N}^\infty	( n \log n )^{ - d  / 2 } \nl
&\lesssim \frac{ L^{2} \log \nnnorm{ \frac x L}_1 N^2 }{ \nnnorm x_L^{d + 2 } }
	+ \frac{ N^{1-d/2} }{L^d (\log N)^{d/2} }
\le \frac{ 2 }{ L^2 \nnnorm x_L^{d - 2 } \log \nnnorm{ \frac x L}_1 } ,
\end{align}
as desired.
 
It remains to prove the other alternative in the minimum in  \eqref{eq:S_bound}--\eqref{eq:S_bound2} for $S_\mu$ when $\mu < 1$. 
We again consider the case $\alpha < 2$ first.
Using \eqref{eq:Dnx} on all terms in \eqref{eq:S_mu}, we have
\begin{equation} \label{eq:ub_pf2}
S_\mu(x) - \delta_{0,x}
\lesssim  \frac{ L^{\alpha } }{ \nnnorm x_L^{d + \alpha } } 
	\sum_{n=1}^\infty n \mu^n
= \frac{ 1 }{ L^{\alpha } \nnnorm x_L^{d - \alpha } \nnnorm{\frac x L}_1^{ 2 \alpha } }
	\frac{\mu}{(1-\mu)^2} ,
\end{equation}
and the desired bound follows from using $\mu \le1$ in the numerator.
The same proof works for $\alpha > 2$ after replacing all $\alpha$ by 2. 
If $\alpha = 2$, we have an extra factor of $\log \nnnorm {\frac x L}_1$ from \eqref{eq:Dnx}, so
\begin{equation}
S_\mu(x) - \delta_{0,x}
\lesssim \frac{ L^{2} \log \nnnorm {\frac x L}_1 }{ \nnnorm x_L^{d + 2 }  }
	\sum_{n=1}^\infty n \mu^n 
= \frac{ 1 }{ L^{2 } \nnnorm x_L^{d - 2 } \log \nnnorm {\frac x L}_1 } 
	 \frac {  (\log \nnnorm {\frac x L}_1)^2 \mu } { { \nnnorm{\frac x L}_1^{ 4 } } (1-\mu)^2 } ,
\end{equation}
and the desired result again follows from $\mu\le1$.
This concludes the proof.
\end{proof}

\section{Proof of Theorem~\ref{thm:G=S}}
\label{sec:G=S}

In this section, we prove Theorem~\ref{thm:G=S}.
Throughout the section, we assume 
$\alpha, d > 0$ satisfy $d > \dc$ if $\alpha \ne 2$, 
or satisfy $d \ge \dc$ if $\alpha = 2$. 
We also assume $D$ (for percolation or self-avoiding walk) or $J$ (for Ising) obeys Assumption~\ref{ass:D}.
As discussed in Section~\ref{sec:strategy},
our main task is to prove the lace expansion equation \eqref{eq:lace} 
and the estimate \eqref{eq:Pi_claim} for $\Pi_p(x)$. 
We employ a bootstrap argument, similar to but simpler than the ones used in \cite{HHS03,CS15,CS19,LS24b}.

\subsection{Bootstrap argument}

The bootstrap argument compares $G_p$ to the upper bound of $C_1$ given in Proposition~\ref{prop:C}. 
Let $K_C$ denote the constant of either \eqref{eq:C_bound} or \eqref{eq:C_bound2}, depending on $\alpha$. 
We define the \emph{bootstrap function} $b: [\half, p_c] \to [0, \infty]$ by
\begin{equation} \label{eq:boot}
b(p) = p \vee
	\sup_{x\neq 0} \frac{G_p(x)}{ K_C L^{-(\alpha \wedge 2)} \nnnorm x_L^{-(d-\alpha \wedge 2)} }  
	\qquad (\alpha \ne 2)
\end{equation}
or by
\begin{equation} \label{eq:boot2}
b(p) = p \vee
	\sup_{x\neq 0} \frac{G_p(x)}{ K_C L^{- 2} \nnnorm x_L^{-(d-2)} ( \log \nnnorm{ \frac x L}_1)\inv }  
	\qquad (\alpha = 2) .
\end{equation}
(Recall our convention that $\nnnorm{ \frac x L }_1 \ge \frac \pi 2 > 1$.)
The function $b(p)$ satisfies $b(1) \le 1$ since $G_1 \le C_1$ \cite[Lemma~2.3]{CS15}, and it is continuous in $p\in [\half, p_c)$ 
by \cite[(3.31)--(3.34)]{CS15} (for $\alpha\ne2$) or \cite[(3.13)--(3.16)]{CS19} (for $\alpha = 2$).
The proof of continuity uses the bound $G_p(x) \le K_{p,L} \nnnorm x_L^{-(d+\alpha)}$ when $p < p_c$, which can be proved via a Simon--Lieb type inequality \cite[Lemma~2.4]{CS15} or via \cite{Aoun21}.
We will prove the implication
\begin{equation}
b(p) \le 3
\quad \implies \quad
b(p) \le 2
\end{equation}
for all $p < p_c$, so that $b(p)$ cannot take any values in the interval $(2,3]$. Continuity then implies that $b(p) \le 2$ for all $p < p_c$. 
Using the left-continuity of $G_p$ at $p_c$, we will also get $b(p_c) \le 2$. 

The \emph{a priori} estimate $b(p)\le 3$ on the bootstrap function gives a control on $G_p(x)$ and allows the lace expansion to be carried out.
For the Ising model,
from $p \le b(p) \le 3$ and \eqref{eq:Ising_beta}
we also get $\beta \le 3 [1 + O(L^{-2d})] \le 4$, so that the $D$ defined in \eqref{eq:D_Ising} obeys Assumption~\ref{ass:D}. 
We collect the lace expansion equation and the diagrammatic estimates in the following proposition, whose proof we postpone to Section~\ref{sec:lace}.
Recall the model-dependent number $\ell \in \{2, 3\}$ defined in \eqref{eq:ell}.

\begin{proposition}[Lace expansion]
\label{prop:lace}
There is an $L_1$ such that for all $L \ge L_1$,
if $b(p) \le 3$ then there exists a $\Z^d$-symmetric function $\Pi_p : \Z^d \to \R$
for which
\begin{equation} \label{eq:lace-boot}
G_p = \delta + p D*G_p + \Pi_p*G_p,
\end{equation}
and
\begin{equation} \label{eq:Pi_bound}
\abs{ \Pi_p(x) }
\lesssim  \frac { \delta_{0,x} } {L^d} + 
\bigg( \frac 1 { L^d \nnnorm{ \frac x L }_1^{d-\alpha \wedge 2 } } \bigg)^\ell 
\times \begin{cases}
1 						& (\alpha \ne 2) \\
( \log \nnnorm {\frac x L}_1 )^{-\ell}	& ( \alpha = 2 ) 
\end{cases} 
\end{equation}
with a constant independent of $p, L, x$.
\end{proposition}

We note that, by considering the two cases $0 < \abs x \le L$ and $\abs x > L$, \eqref{eq:Pi_bound} implies that
\begin{equation} \label{eq:Pi_L1}
\sum_{x\ne 0} \abs{ \Pi_p(x) }
\lesssim \sum_{ 0 < \abs x \le L } \frac 1 { L^{\ell d} }
	+  \sum_{ \abs x > L } \Bigl( \frac  1 { L^{\alpha \wedge 2} \abs x^{d-\alpha \wedge 2 } } \Bigr)^\ell	
\lesssim \frac 1 { L^{(\ell-1)d} } .
\end{equation}
Given $b(p) \le 3$, 
\eqref{eq:lace-boot} and \eqref{eq:Pi_bound} allow us to carry out the strategy of Section~\ref{sec:strategy}:

\begin{proposition} \label{prop:tilde_D}
Let $0 < \alpha \le 2 + (\ell - 1)(d - \dc)$. 
There is an $L_2 \ge L_1$ such that for all $L\ge L_2$,
if $b(p) \le 3$ then 
the distribution $\tildeD$ of \eqref{eq:tilde_D} is well-defined
and equation \eqref{eq:mup} holds.
Moreover, we have
\begin{equation} \label{eq:perturb-boot}
A_p, \,    \frac \mup p = 1 + O(L^{-d}) ,
\qquad
\frac{ \tildeD(x) }{ D(x) },\,
\frac{ 1 - \hat{ \tildeD }(k) } { 1 - \hat D(k) }
= 1 + O(L^{-(\ell-1) d}) ,
\end{equation}
with constants uniform in $p \in [\half, p_c ]$, in $x\ne0$, and in $k \ne 0$.
\end{proposition}

\begin{proof}
Let $L \ge L_1$ and $b(p) \le 3$, so that we can apply Proposition~\ref{prop:lace}.
We have noted in \eqref{eq:decay_comp} that
\begin{equation}
\ell( d - \alpha) \ge d + \alpha
	\qquad (\alpha \le 2,\ d\ge \dc) ,
\end{equation}
and in \eqref{eq:alpha_cond} that
\begin{equation}
\ell( d-2) \ge d+\alpha
	\qquad ( 2 < \alpha \le 2 + (\ell-1)(d-\dc) ) ,
\end{equation}
so we always have
\begin{equation}
\ell( d - \alpha \wedge 2) \ge d + \alpha .
\end{equation}
Using \eqref{eq:Pi_bound} and the lower bound of $D(x)$ in \eqref{eq:D_bounds}, we have
\begin{equation} \label{eq:PiD_pf}
\abs{ \Pi_p(x) }
\lesssim \frac 1 {L^{\ell d}  \nnnorm{\frac x L}_1^{\ell(d - \alpha \wedge 2)} }
\le \frac 1 {L^{\ell d}  \nnnorm{\frac x L}_1^{d+\alpha} }
\lesssim \frac 1 {L^{(\ell-1)d}} D(x)
	\qquad (x\ne 0).
\end{equation}
By restricting to $L$ large and $p\ge \half$, we can therefore ensure $p D(x) + \Pi_p(x) \ge 0$ and define the probability distribution $\tildeD$ by \eqref{eq:tilde_D}.
Equation \eqref{eq:mup} follows. 

It remains to prove \eqref{eq:perturb-boot}.
For $A_p$ and $\mup/p$, we use the \eqref{eq:Pi_bound}, \eqref{eq:Pi_L1}, and $p\ge \half$ to get
\begin{align}
A_p &= [ 1 - \Pi_p(0) ]\inv = 1 + O(L^{-d}), 
\\
\frac \mup p 
&= A_p \Big( 1 + \frac 1 p \sum_{y\ne 0} \Pi_p(y) \Big)
= A_p [ 1 + O(L^{-(\ell-1) d}) ] 
= 1 + O(L^{-d}) .
\end{align}
For $\tildeD(x) / D(x)$, 
we use the definition of $\tildeD$ in \eqref{eq:tilde_D}, $p\ge\half$, and \eqref{eq:Pi_L1}
to write
\begin{equation}
\tildeD(x) 
= \frac{ D(x) + \frac 1 p \Pi_p(x) } { 1 + \frac 1 p \sum_{y\ne 0} \Pi_p(y) }
= [ 1 + O( L^{-(\ell -1)d}) ] [ D(x) + \frac 1 p \Pi_p(x) ]
	\qquad(x\ne0).
\end{equation}
Then we divide the equation by $D(x)$ and get the desired result from \eqref{eq:PiD_pf}.
For the Fourier transform,
we use $\Zd$-symmetry and the uniform estimate on $\tildeD(x) / D(x)$ to get
\begin{equation}
1 - \hat{ \tildeD}(k)
=  \sum_{x\ne0} \Bigl( 1 + \frac{ \tildeD(x) }{ D(x) } -1 \Bigr) 
	D(x) ( 1 - \cos(k\cdot x))
= [ 1 + O(L^{-(\ell-1) d}) ] ( 1 - \hat D(k) ) .
\end{equation}
This concludes the proof.
\end{proof}

\begin{lemma}[Bootstrap]
\label{lem:boot}
Let $0 < \alpha \le 2 + (\ell - 1)(d - \dc)$. 
There is an $L_3 \ge L_2$ such that for all $L \ge L_3$, 
if $p \in [\half,p_c)$ and $b(p) \le 3$ 
then $\mup < 1$ and $b(p) \le 2$.
\end{lemma}

\begin{proof}
Let $L \ge L_2$ and $p \in [\half, p_c)$. 
If $b(p) \le 3$, we can apply Proposition~\ref{prop:lace},
sum equation~\eqref{eq:lace-boot} over $x\in \Zd$, 
and then divide by $\chip = \sum_{x\in \Zd} G_p(x) < \infty$, to get
\begin{equation} \label{eq:subcrit}
1 = \frac 1 \chip + p + \sum_{x\in \Zd} \Pi_p(x) .
\end{equation}
Since $\chip \in [1,\infty)$, we obtain $1 > p + \sum_{x\in \Zd} \Pi_p(x)$, 
which is equivalent to $\mup < 1$ by definition \eqref{eq:mup}.
Using \eqref{eq:Pi_bound} and \eqref{eq:Pi_L1}, from the same inequality we also get $p \le 1 + O( L^{-d})$.

To prove $b(p) \le 2$, it remains to estimate $G_p$.
We first use Proposition~\ref{prop:tilde_D} to get equation~\eqref{eq:mup}, and then we observe that both $[1-\Pi_p(0)]G_p$ and the random walk two-point function $S_\mup$ generated by $\tildeD$ are $\ell^1$ solutions to the convolution equation
\begin{equation} \label{eq:unique1}
S = \delta + \mup \tildeD * S .
\end{equation}
It then follows from the uniqueness of solutions (the Fourier transform of the difference of two solutions is $0$) that $[1 - \Pi_p(0)] G_p = S_\mup$ everywhere. 
We use Proposition~\ref{prop:S} on $\tildeD$ to estimate $S_\mup$,
and we verify the hypotheses \eqref{eq:P_bounds}--\eqref{eq:Phat} using \eqref{eq:perturb-boot}, Assumption~\ref{ass:D}, and by restricting to $L$ sufficiently large. 
This gives
\begin{equation}
[1 - \Pi_p(0)] G_p (x)
= S_\mup (x)
\le S_1 (x)
\le \frac{ K_C }{ L^{\alpha\wedge 2} \nnnorm x_L^{d - \alpha \wedge 2 } }
\times
	\begin{cases}
	1 							& (\alpha \ne 2) \\
	( \log \nnnorm {\frac x L}_1 )\inv	& ( \alpha = 2 ) 
	\end{cases} 
\end{equation}
uniformly in $x\ne 0$.
Putting the above into the definition of $b(p)$, 
and using $\Pi_p(0) = O(L^{-d})$,
we obtain $b(p) \le 1 + O(L^{-d}) \le 2$
for $L$ sufficiently large.
This concludes the proof.
\end{proof}

\begin{proof}[Proof of Theorem~\ref{thm:G=S}]
Set $L_0 = L_3$. 
We recall the bootstrap function $b(p)$ is continuous in $p \in [\half, p_c)$, 
satisfies $b(1) \le 1$, 
and cannot take any values in $(2, 3]$ when $p \in [\half, p_c)$ by Lemma~\ref{lem:boot}.
Therefore, by continuity we must have $b(p) \le 2$ for all $p \in [\half, p_c)$.
Proposition~\ref{prop:tilde_D} and the uniqueness of solution argument
given around \eqref{eq:unique1} then give \eqref{eq:G=S} and \eqref{eq:perturb} for $p \in [\half,p_c)$.

For the statements at $p_c$,
we first use the uniform bound $b(p) \le 2$ established above and 
the left-continuity of $G_p(x)$ at $p = p_c$ \cite[Lemma~2.2]{CS15} to get
\begin{equation} \label{eq:G_bound}
G\crit(x)
= \lim_{p \to p_c^-} G_p (x)
\le \frac{ 2 K_C  }{ L^{\alpha\wedge 2} \nnnorm x_L^{d - \alpha \wedge 2 } }
\times
	\begin{cases}
	1 							& (\alpha \ne 2) \\
	( \log \nnnorm {\frac x L}_1 )\inv	& ( \alpha = 2 ) 
	\end{cases} 
\end{equation}
for all $x\ne 0$.
This implies $b(p_c) \le 2$, so Proposition~\ref{prop:tilde_D} still applies and gives \eqref{eq:mup} and \eqref{eq:perturb} at $p_c$. But, to prove \eqref{eq:G=S}, we need to modify the uniqueness argument slightly and use the $L^2$ Fourier transform.
We first observe that \eqref{eq:G_bound} implies $G\crit \in \ell^2(\Zd)$, since $d > \dc$ (for $\alpha \ne 2$) or $d \ge \dc$ (for $\alpha =2$).
Similarly, Proposition~\ref{prop:S} implies $S_1 \in \ell^2(\Zd)$ in these dimensions. 
Assume $\mupc = 1$ for a moment, 
by \eqref{eq:mup} we find that both $[1-\Pi\crit(0)]G\crit$ and $S_1$ are $\ell^2$ solutions to
\begin{equation}
S = \delta + \tildeD * S ,
\end{equation}
so we can use the uniqueness of $\ell^2$ solutions (the $L^2$ Fourier transform of the difference of two solutions is $0$) to conclude $[1-\Pi\crit(0)]G\crit = S_1$, which gives \eqref{eq:G=S}.

To prove $\mupc = 1$,
we rearrange equation \eqref{eq:lace-boot} at $p_c$ as
\begin{equation}
(\delta - p_c D - \Pi\crit) * G\crit = \delta .
\end{equation}
This allows us to apply Proposition~\ref{prop:crit} with $F = \delta - p_c D - \Pi\crit \in \ell^1(\Zd)$ and $G = G\crit \in \ell^2(\Zd)$. 
Since $\sum_{x\in \Zd} G\crit(x) = \chi(p_c)= \infty$, we obtain $\sum_{x\in \Zd} F(x) = 1 - p_c - \sum_{x\in \Zd} \Pi\crit(x) = 0$, which is equivalent to $\mupc = 1$ by definition \eqref{eq:mup}.

It remains to prove \eqref{eq:near_crit_claim}.
Using a standard convolution estimate \cite[Lemma~3.5]{CS19},
the bound \eqref{eq:G_bound} implies
that the bubble diagram $B(p_c) = G\crit * G\crit(0)$ for Ising and SAW, 
or the triangle diagram $T(p_c) = G\crit^{*3}(0)$ for percolation,
is $1 + O(L^{-d}) < \infty$. 
It follows that $\chip \asymp (p_c- p) \inv$ \cite{BFF84, Slad06, AN84, Aize82}. 
By summing equation \eqref{eq:G=S} over $x\in \Zd$, we also have
\begin{equation}
\chip = \sum_{x\in \Zd} G_p(x) 
= A_p \sum_{x\in \Zd} S_\mup (x)
= \frac {A_p }{ 1 - \mup } 
= \frac { 1+ O(L^{-d } ) }{ 1 - \mup } .
\end{equation}
Comparing the two equations gives \eqref{eq:near_crit_claim}, as desired.
This concludes the proof.
\end{proof}

\subsection{Proof of Proposition~\ref{prop:lace}}
\label{sec:lace}

It remains to prove Proposition~\ref{prop:lace}.
The desired equation \eqref{eq:lace-boot} is the lace expansion equation for self-avoiding walk \cite{BS85,Slad06}. The bound \eqref{eq:Pi_bound} on $\Pi_p$ is known as a diagrammatic estimate, and it is proved by applying convolution estimates to lace expansion diagrams, first done in \cite{HHS03}. 
For the self-avoiding walk, 
Proposition~\ref{prop:lace} is proved in  \cite[Proposition~3.1]{CS15} (for $\alpha \ne 2$) and \cite[Lemma~3.6]{CS19} (for $\alpha=2$) (both denote our $\Pi_p$ by $\pi_p$), 
by improving the convolution estimates of \cite{HHS03} using $\nnnorm \cdot_L$. 
We therefore focus on percolation and the Ising model. 

The lace expansion equation for percolation \cite{HS90a} and the Ising model \cite{Saka07} takes the form
\begin{equation} \label{eq:lace_perc}
G_p = h_p + p D * h_p * G_p ,
\end{equation}
where $h_p$ is a perturbation of the Kronecker delta $\delta(x) = \delta_{0,x}$.
It is observed in \cite{LS24b} that equation \eqref{eq:lace_perc} can be transformed into the desired \eqref{eq:lace-boot} in the spread-out setting. 
The proof in \cite{LS24b} uses a Banach algebra (as in \cite{BHK18}),
and the main idea is to convolve equation \eqref{eq:lace_perc} with the convolution inverse of $h_p$, defined by the Neumann series
\begin{equation} \label{eq:Neumann}
h_p\inv 
= ( \delta - (\delta - h_p) )\inv
= \sum_{n=0}^\infty (\delta - h_p)^{*n} .
\end{equation}
This gives $ h_p\inv * G_p = \delta + p D* G_p$, so
\begin{equation} \label{eq:lace_pf}
G_p = \delta + pD * G_p + ( \delta - h_p\inv ) * G_p ,
\end{equation}
which is \eqref{eq:lace-boot} with $\Pi_p = \delta - h_p \inv = - \sum_{n=1}^\infty (\delta - h_p)^{*n} $. 
We simply estimate the series \eqref{eq:Neumann}, using the following lemma.
Note that both $f_1,f_2$ are integrable since $\theta > 0$, so the upper bound on $f_1 * f_2$ takes the same form. 

\begin{lemma} \label{lem:conv}
Let $d, L \ge 1$, $\theta > 0$, and $\phi, \psi \ge 0$.
Suppose $f_1,f_2: \Zd \to \R$ obey the bound
\begin{equation}
\abs{ f_1(x) }, \abs{f_2(x)}
\le \delta_{0,x} + \frac 1 { L^{d+\phi} \nnnorm{ \frac x L }_1^{d+\theta} 
	( \log \nnnorm{\frac x L}_1 )^{\psi} } 
\qquad (x\in \Zd) .
\end{equation}
Then there is a constant $C= C(d,\theta,\psi) > 0$ such that
\begin{equation}
\abs{ (f_1*f_2)(x) }
\le C \bigg( \delta_{0,x} + \frac 1 { L^{d+\phi} \nnnorm{ \frac x L }_1^{d+\theta}
	( \log \nnnorm{\frac x L}_1 )^{\psi} } \bigg)
\qquad (x\in \Zd) .
\end{equation}

\end{lemma}

\begin{proof}
For $j \in \{1,2\}$,
define $g_j = f_j - f_j(0)\delta$, so that $g_j(0) = 0$ and $g_j(x) = f_j(x)$ when $x\ne 0$.
Since
\begin{equation}
f_1*f_2 = f_1(0)f_2(0) \delta + f_1(0) g_2 + f_2(0)g_1 + g_1*g_2 ,
\end{equation}
and since $\abs {f_j(0)} \le 1 + L^{ -(d + \phi)} \le 2$, 
we only need to estimate $g_1*g_2$.
By the hypothesis, we have
\begin{align}
\abs{ g_1*g_2(x) }
&\le \sum_{y\in \Zd} \frac1{ L^{d+\phi}\nnnorm{ \frac {x-y} L }_1^{d+\theta} ( \log \nnnorm{\frac {x-y} L}_1 )^{\psi}} 
	\frac 1 { L^{d+\phi} \nnnorm{ \frac y L }_1^{d+\theta} ( \log \nnnorm{\frac y L}_1 )^{\psi}}   \nl
&= \frac {L^{2\theta}} {L^{2\phi}} \sum_{y\in \Zd} \frac 1{ \nnnorm{ x-y }_L^{d+\theta} ( \log \nnnorm{\frac {x-y} L}_1 )^{\psi} } 
\frac 1 { \nnnorm{ y }_L^{d+\theta} ( \log \nnnorm{\frac y L}_1 )^{\psi}}   \nl
&\le \frac {L^{2\theta}} {L^{2\phi}} \frac{ C_{d,\theta,\psi} L^{-\theta} }{ \nnnorm x_L^{d+\theta} ( \log \nnnorm{\frac x L}_1 )^{\psi} }
= \frac {C_{d,\theta,\psi}} {L^{\phi}} \frac{ 1 }{ L^{d+\phi} \nnnorm{ \frac x L }_1^{d+\theta} ( \log \nnnorm{\frac x L}_1 )^{\psi} } ,
\end{align}
where in the last line we used the convolution estimate \cite[Lemma~3.5]{CS19} with $a_1 = b_1 = d + \theta$ and $a_2 = b_2 = \psi$ (or use \cite[Lemma~3.2(i)]{CS15} if $\psi=0$).
This gives the desired estimate, since $\phi \ge 0$.
\end{proof}

\begin{proof}[Proof of Proposition~\ref{prop:lace} for percolation and the Ising model]
It is proved in \cite[Proposition~3.1]{CS15} (for $\alpha \ne 2$) and \cite[Lemma~3.6]{CS19} (for $\alpha=2$) (both denote our $h_p$ by $\pi_p$) that, under $b(p) \le 3$, 
equation \eqref{eq:lace_perc} holds with a function $h_p$ that obeys
\begin{equation} \label{eq:h_bound}
\abs{ h_p(x) - \delta_{0,x} }
\le K_h  \frac { \delta_{0,x} } {L^d} + 
K_h \bigg( \frac 1 { L^d \nnnorm{ \frac x L }_1^{d-\alpha \wedge 2 } } \bigg)^\ell 
\times \begin{cases}
1 						& (\alpha \ne 2) \\
( \log \nnnorm {\frac x L}_1 )^{-\ell}	& ( \alpha = 2 ) ,
\end{cases} 
\end{equation}
where $K_h > 0$ is independent of $p,L,x$.
(See also \cite[Proposition~1.8]{HHS03} for percolation and \cite[Corollaries~5.5, 5.10, and 5.13]{Saka07_correction} for the Ising model.)

Set $\phi = (\ell-2) d$, $\psi = \ell \1_{\alpha = 2}$, and
\begin{equation}
\theta = \ell( d - \alpha \wedge 2) - d 
\ge \alpha \wedge 2 > 0,
\end{equation}
where the middle inequality is due to $d \ge \dc = \frac{\ell + 1 }{ \ell - 1} (\alpha\wedge2)$.
We first apply Lemma~\ref{lem:conv} to $f_1 = f_2 = \frac{L^d}{K_h}( h_p - \delta )$, then a simple induction gives
\begin{equation} \label{eq:Neumann_pf}
\abs{ ( h_p - \delta )^{*n} (x) }
\le C^{n-1} \bigg( \frac{K_h}{L^d}\bigg)^n
	\bigg( \delta_{0,x} + \frac 1 { L^{(\ell-1)d} \nnnorm{ \frac x L }_1^{ \ell(d-\alpha \wedge 2)}
	( \log \nnnorm{\frac x L}_1 )^{\psi} } \bigg)
\end{equation}
for all $n\ge 1$.
By restricting to $L$ large enough such that $CK_h L^{-d} < 1$, 
we can sum \eqref{eq:Neumann_pf} over $n\ge1$ to get
\begin{equation} \label{eq:Neumann_sum}
\biggabs{ \sum_{n=1}^\infty (\delta - h_p)^{*n} (x) }
\le \frac{K_h L^{-d} }{ 1 - CK_h L^{-d} } \bigg( \delta_{0,x} + \frac 1 { L^{(\ell-1)d} \nnnorm{ \frac x L }_1^{ \ell(d-\alpha \wedge 2)}
	( \log \nnnorm{\frac x L}_1 )^{\psi} } \bigg) .
\end{equation}
It follows that the Neumann series \eqref{eq:Neumann} defining $h_p\inv$ converges, and from equation \eqref{eq:lace_pf} we get the desired \eqref{eq:lace-boot}.
The desired estimate on $\Pi_p = \delta - h_p\inv = - \sum_{n=1}^\infty (\delta - h_p)^{*n}$ also follows from \eqref{eq:Neumann_sum}.
This concludes the proof.
\end{proof}

\appendix

\section{Characterisation of critical point}

In this appendix, we develop a tool used in the proof of Theorem~\ref{thm:G=S} to conclude that $\mupc = 1$, or equivalently, that
\begin{equation} \label{eq:pc}
1 =  p_c + \sum_{x\in \Zd} \Pi\crit(x) .
\end{equation}
This characterisation of the critical point $p_c$ is traditionally obtained by taking the $p \to p_c^-$ limit of the subcritical version of \eqref{eq:pc} presented in \eqref{eq:subcrit}. 
The traditional way, however, uses the left-continuity of $\Pi_p$ at $p_c$, which for some models is not so straightforward to prove (see \cite[Appendix~A]{Hara08} for percolation\footnote{
As pointed out by a referee, Hara's proof for percolation also works for the Ising model, as the lace expansion coefficients are defined in terms of the same events.
} and the self-avoiding walk).
We work directly at $p_c$ and avoid using the left-continuity of $\Pi_p$.

We consider functions $F, G: \Zd \to \R$ that satisfy
\begin{equation} \label{eq:FG}
(F*G)(x) = \delta_{0,x}
	\qquad(x\in \Zd) .
\end{equation}
If both $F, G\in \ell^1(\Zd)$, or if both $F,G\ge 0$, by summing \eqref{eq:FG} over all $x\in \Zd$ and using Fubini's theorem, we have 
\begin{equation} 
\sum_{x\in \Zd} F(x)  \sum_{x\in \Zd} G(x)  = 1.
\end{equation}
The main result of this appendix is an extension of the above to the case where $G\in \ell^2 \setminus \ell^1(\Zd)$.

\begin{proposition} \label{prop:crit}
Let $d\ge 1$.
Suppose $F \in \ell^1(\Zd)$, 
$G\in \ell^2(\Zd)$, $G\ge 0$,
and suppose they satisfy equation \eqref{eq:FG}.
Then
\begin{equation}
\sum_{x\in \Zd} G(x) = \infty
\quad \implies \quad
\sum_{x\in \Zd} F(x) = 0.
\end{equation}
\end{proposition}

The proof uses the \emph{Bochner--Riesz operator} in the Fourier space.
We write $\T^d=(\R/2\pi\Z)^d$ for the continuum torus, 
and we use the (inverse) Fourier transform
\begin{equation}
\hat f(k)  = \sum_{x\in\Z^d}f(x) e^{ik\cdot x} 	\quad (k \in \T^d),
\qquad
\check u(x) = \int_{\T^d} u(k) e^{-ik\cdot x}  \frac{ dk }{ (2\pi)^d } \quad (x\in \Z^d)
\end{equation}
for $f \in \ell^1(\Zd)$ and $u\in L^1(\Td)$.
For $R>0$, $\alpha \ge 0$, and $u \in L^1(\Td)$, we define
\begin{equation}
B_R^\alpha[u](k) 
= \sum_{\abs x \le R} \biggl( 1 - \frac{ \abs x^2 }{R^2} \biggr)^\alpha \check u(x) e^{ik\cdot x}
	\qquad (k\in \Td) .
\end{equation}
The Bochner--Riesz operator $B_R^\alpha$ is an approximate identity:

\begin{lemma} \label{lem:BR}
Let $d\ge 1$.
If $\alpha > \frac{d-1}2$
and if the function $u \in L^1(\Td)$ is continuous at $k_0 \in \Td$,
then
\begin{equation}
\lim_{R\to \infty} B_R^\alpha[u](k_0) 
= u(k_0) .
\end{equation}
\end{lemma}

\begin{proof}[Proof of Proposition~\ref{prop:crit} assuming Lemma~\ref{lem:BR}]
We use the Fourier transform. 
The hypotheses imply that $\hat F$ is continuous on $\Td$ and
that $\hat G \in L^2(\Td) \subset L^1(\Td)$. 
From equation \eqref{eq:FG} we get
$\hat F(k) \hat G(k) = 1$
for almost every $k\in \Td$. 
Our goal is to prove $\hat F(0) = 0$. We prove by contradiction.

Suppose $\hat F(0) \ne 0$.
Since $\hat F$ is continuous, the function $1 / \hat F(k)$ is defined and continuous in a neighbourhood of $0$.
We apply Lemma~\ref{lem:BR} with $u = 1 / \hat F = \hat G \in L^1(\Td)$, $k_0 = 0$, and with any choice of $\alpha > \frac{d-1} 2$, to get
\begin{equation}
\lim_{R\to \infty} B_R^\alpha[1 / \hat F](0) 
= \frac 1 { \hat F(0) } .
\end{equation}
On the other hand, 
since the inverse Fourier transform of $1/ \hat F = \hat G$ is $G$, the definition of $B_R^\alpha$ gives
\begin{equation}
B_R^\alpha [ 1 /  \hat F ] (0)
= B_R^\alpha [ \hat G ] (0)
= \sum_{\abs x \le R} \biggl( 1 - \frac{ \abs x^2 }{R^2} \biggr)^\alpha G(x)
\to  \sum_{x \in \Zd} G(x) = \infty ,
\end{equation}
by monotone convergence.
This is a contradiction to $\hat F(0) \ne 0$,
so the proof is complete.
\end{proof}

It remains to prove Lemma~\ref{lem:BR}. 
The lemma is a slight variant of \cite[Proposition~4.1.9]{Graf14}, 
which states that $B_R^\alpha[u]$ converges uniformly to $u$ if $u$ is continuous on all of $\Td$. 
In that proof, it is shown that, when $\alpha > \frac{d-1}2$, there is a family of kernels $L^{\alpha, R} : \Td \to \C$ for which
\begin{equation} \label{eq:kernel}
B_R^\alpha[u] = u * L^{\alpha, R}
	\qquad (u\in L^1(\Td)) ,
\end{equation}
with the properties that\footnote{
Properties \eqref{eq:kernel_L1}--\eqref{eq:delta_L1} make the family $\{ L^{\alpha, R} \}_{R>0}$ an \emph{approximate identity}, also known as a \emph{summability kernel}.
}
\begin{align}
\label{eq:kernel_L1}
\int_\Td \abs{  L^{\alpha, R}(k) } dk &\le C_{d,\alpha} , 
\\
\label{eq:kernel_mean}
\int_\Td  L^{\alpha, R}(k)  dk &= 1 , 
\\
\label{eq:delta_L1}
\int_{ \delta \le \norm k_\infty \le \pi } \abs{  L^{\alpha, R}(k) } dk 
	&\le C_{d,\alpha, \delta} R^{-(\alpha - \frac{d-1}2)} 
\end{align}
for all $R > 0$ and all $0 < \delta < \pi$. 
(\cite{Graf14} uses the torus with period $1$ rather than our $2\pi$.)
The proof of \eqref{eq:delta_L1} also yields the stronger statement that
\begin{equation} \label{eq:delta_sup}
\sup_{ \delta \le \norm k_\infty \le \pi } \abs{  L^{\alpha, R}(k) } 
	\le C_{d,\alpha, \delta}' R^{-(\alpha - \frac{d-1}2)} 
\end{equation}
for all $R > 0$.

\begin{proof}[Proof of Lemma~\ref{lem:BR}]
We modify the proof of \cite[Theorem~1.2.19]{Graf14}.
Let $u \in L^1(\Td)$ be continuous at $k_0 \in \Td$,
so that for any $\eps > 0$ there is a $\delta > 0$ for which $\norm k_\infty < \delta$ implies $\abs{ u(k_0 - k) - u(k_0) } < \eps$.
We use \eqref{eq:kernel} and \eqref{eq:kernel_mean} to write
\begin{equation}
B_R^\alpha[u](k_0) - u(k_0)
= \int_\Td  L^{\alpha,R}(k)  [ u(k_0 - k) - u(k_0) ] dk .
\end{equation}
By splitting the integral according to how $\norm k_\infty$ compares with $\delta$, and by using \eqref{eq:delta_sup},
we get
\begin{equation}
\abs{ B_R^\alpha[u](k_0) - u(k_0) }
\le  \eps \norm{L^{a,R} }_1 
	+ O(R^{-(\alpha - \frac{d-1}2)}) \big( \norm u_1 + (2\pi)^d \abs{ u(k_0) } \big) .
\end{equation}
Sending $R \to \infty$, \eqref{eq:kernel_L1} then gives
\begin{equation}
\limsup_{R\to\infty}\, \abs{ B_R^\alpha[u](k_0) - u(k_0) }
\le C_{d,\alpha} \eps ,
\end{equation}
and we are done because $\eps > 0$ is arbitrary.
\end{proof}

\section*{Acknowledgements}
The work was supported in part by 
the National Natural Science Foundation of China (Grant No.~12595284, 12595280) 
and NSERC of Canada.
We thank Noe Kawamoto, Romain Panis, and Gordon Slade for comments on a preliminary version.
We thank the anonymous referees for their careful reading and constructive comments.

%

{\small

}
\end{document}